\documentclass[envcountsame]{llncs}

\usepackage{times,amsmath,amssymb,enumerate,gastex,stmaryrd,rotating,mathrsfs}

\newcommand{\links}{[\![}
\newcommand{\rechts}{]\!]}

\begin{document}

\title{Tilings and Submonoids of Metabelian Groups}
\author{Markus Lohrey\inst{1} \and Benjamin Steinberg\inst{2,}\thanks{The authors would like to
acknowledge the support of DFG Mercator program.  The second author is also supported by an NSERC grant.}
\institute{Universit\"at
Leipzig, Institut f\"ur Informatik, Germany
 \and
School of Mathematics and Statistics,
Carleton University, ON, Canada\\
\email{lohrey@informatik.uni-leipzig.de,
bsteinbg@math.carleton.ca}}}

\maketitle

\begin{abstract}
In this paper we show that membership in finitely generated submonoids is
undecidable for the free metabelian group of rank $2$ and for the wreath
product $\mathbb Z\wr (\mathbb Z\times \mathbb Z)$.  We also show that
subsemimodule membership is undecidable for finite rank free
$(\mathbb Z\times \mathbb Z)$-modules.  The proof involves an encoding
of Turing machines via tilings.  We also show that rational subset
membership is undecidable for two-dimensional lamplighter groups.
\end{abstract}

\section{Introduction}
Two of the classical group theoretic decision problems are  the word
problem and the generalized word problem.  Suppose $G$ is a finitely generated
group with finite generating set $\Sigma$ and put $\Sigma^{\pm}=\Sigma\cup
\Sigma^{-1}$.  Let $\pi\colon (\Sigma^{\pm})^*\to G$ be the canonical
projection from the free monoid on $\Sigma^{\pm}$ onto $G$.
The word problem asks to determine algorithmically given an input word $w\in
(\Sigma^{\pm})^*$, whether $\pi(w)=1$.  An algorithm for the (uniform)
generalized word problem takes as input finitely many words
$w,w_1,\ldots, w_n\in (\Sigma^{\pm})^*$ and answers whether
$\pi(w)\in \langle \pi(w_1),\ldots,\pi(w_n)\rangle$.  Two more general
problems that have received some attention in recent years are
the submonoid membership \cite{LohSte08,MaMeSu05} and the rational subset membership
problems~\cite{Gilman96,KaSiSt06,LohSte08,Ned00,Rom99}.

The (uniform) submonoid membership problem for $G$ takes as input a
finite list of words $w,w_1,\ldots, w_n\in (\Sigma^{\pm})^*$ and asks the
question is $\pi(w)\in \{\pi(w_1),\ldots,\pi(w_n)\}^*$ (where if $X\subseteq
G$, then $X^*$ denotes the submonoid generated by $G$).  For example, $g\in G$ has
finite order if and only if $g^{-1}\in g^*$ and so decidability of membership
in cyclic submonoids allows one to compute the order of an element.  Of
course, decidability of submonoid membership implies decidability of
the generalized word problem.  In~\cite{LohSte08} the authors provided the first
example of a group with decidable generalized word problem and undecidable
submonoid membership problem, namely the right-angled Artin group
(or graph group) whose associated graph is a path of length $3$.

The rational subset membership problem for $G$ is the following algorithmic
problem:  given as input a word $w\in(\Sigma^{\pm})^*$ and a finite automaton
$\mathscr A$ over $(\Sigma^{\pm})^*$, determine whether
$\pi(w)\in \pi(L(\mathscr A))$.  Of course, this is the most general
of the problems we have been discussing, and is therefore the
easiest to prove undecidable.  A subset of the group $G$ is called
\emph{rational} if it is of the form $\pi(L(\mathscr A))$ for
some finite automaton $\mathscr A$.  This definition does not
depend on the generating set.   The study of rational subsets of groups goes back a long way, beginning with free groups~\cite{Benois69} and commutative groups~\cite{EiSchu69}.  Other early references include~\cite{AniSeif75,Ber79}. The largest known class of groups with decidable rational subset membership problem can be found in~\cite{LohSte08}, where one also finds a complete classificaiton of graph groups with decidable rational subset membership problem.

It is easy to see that if $L_1,L_2$ are rational subsets of $G$,
then $L_1\cap L_2\neq \emptyset$ if and only if $1\in L_1L_2^{-1}$
and the latter product is a rational subset of $G$.  There are many
monoids embeddable in groups with undecidable rational subset
intersection emptiness problems; one such example was exploited in~\cite{LohSte08}.
An easier example is the following.  Let $M$ be a free monoid on two-generators.
Then the undecidability of the Post correspondence problem implies
that there is a fixed finitely generated submonoid $N$ of $M\times M$
so that it is undecidable given an element $(u,v)\in M\times M$ whether
$(u,v)N\cap \Delta\neq \emptyset$ where $\Delta$ is the diagonal submonoid
of $M\times M$.  Therefore, if $G$ is a finitely generated group containing
$M\times M$, then $(u,v)N\cap \Delta\neq \emptyset$ if and only if
$(u,v)\in \Delta N^{-1}$ and so rational subset membership is undecidable
for a fixed rational subset of such a group.  For instance, Thompson's group
$F$ contains a direct product of two free monoids of rank $2$ and so
has undecidable rational subset membership problem.
It was shown by \cite{Rom99} that, for any nilpotency class $c\geq 2$,
there is a rank $r$ so that the free nilpotent group of class $c$ and
rank $r$ has undecidable rational subset membership problem
via an encoding of Hilbert's tenth problem. On the other hand, from
the subgroup separability of polycyclic groups \cite{Mal83} it follows that the latter
have a decidable generalized word problem. A more practical algorithm
can be found in \cite{AvWi89}.

It is well known that free solvable groups of rank $2$ and derived length at
least $3$ have undecidable generalized word problem \cite{Umi95}.
On the other hand, all finitely generated metabelian groups have a decidable generalized
word problem~\cite{Rom74,Rom80}.  It is therefore natural to consider metabelian
groups for the submonoid and rational subset membership problems.
In this paper we show that there is a fixed finitely generated submonoid
of the free metabelian group of rank $2$ with undecidable membership problem.
The same result is also established for the wreath product
$\mathbb Z\wr (\mathbb Z\times \mathbb Z)$.  The proof is via a
reduction to the membership problem for finitely generated subsemimodules
of free $(\mathbb Z\times \mathbb Z)$-modules of finite rank.
This latter problem we prove undecidable by interpreting it as
a particular tiling problem that we show to be undecidable
via a direct encoding of a Turing machine.

The paper ends by showing that membership in rational subsets of
the metabelian group
$\mathbb Z/n\mathbb Z\wr (\mathbb Z\times \mathbb Z)$ is undecidable
using essentially the same tiling problem.  It is left open whether this group
has a decidable submonoid membership problem.  At the moment, there
are no examples of groups for which the submonoid membership problem
is decidable, but the rational subset membership problem is
undecidable. Some further algorithmic results concerning metabelian groups
can be found in \cite{BaCaRoDe94,MyRoUsVe08,Rom79}.

\section{The subsemimodule membership problem}
Fix a group $G$.  Recall that a (left) \emph{$G$-module} is an abelian group $M$
equipped with a left action of $G$ by  automorphisms.  Equivalently, a
$G$-module is a module for the group ring $\mathbb ZG$.  One can extend
this definition to obtain the notion of a $G$-semimodule.
By a \emph{$G$-semimodule}, we mean a commutative monoid $N$
equipped with a left action of $G$ by automorphisms.  Equivalently,
we are speaking of semimodules for the group semiring $\mathbb NG$.
If $X$ is a subset of a $G$-module, then $\mathbb ZG\cdot X$ will
denote the submodule generated by $X$ and $\mathbb NG\cdot X$ will denote the subsemimodule generated by $X$.

Let us now formulate the membership problem for semimodules.
Informally, the problem is given a fixed finitely generated $G$-module
$M$, can one determine algorithmically membership in finitely generated
subsemimodules of $M$.  Of course, to make this a well-defined algorithmic
problem we need to describe how to represent elements of the module.

Assume now that our group $G$ is generated by a finite set
$\Sigma$ and denote by $\Sigma^{\pm}$ the set $\Sigma\cup \Sigma^{-1}$.
Let $\mathbb Z\Sigma^{\pm}$ be the ring of integral polynomials
in non-commuting variables $\Sigma^{\pm}$ (that is the free ring on
$\Sigma^{\pm}$).
There is as usual a canonical surjection
$\pi\colon \mathbb Z\Sigma^{\pm}\to \mathbb ZG$ induced by evaluating words in $G$.

Let $M$ be a finitely generated $G$-module with generating set $B$.
We can view it as a $\mathbb Z\Sigma^{\pm}$-module via $\pi$.
Let $\widetilde M$ be the free  $\mathbb Z\Sigma^{\pm}$-module
on $B$.
Then there is a canonical projection $\rho\colon \widetilde M\to M$ sending
$B$ to $B$.
The idea then is that we can represent elements of $M$ by elements of
$\widetilde M$.
The (uniform) \emph{subsemimodule membership problem} then takes as input a finite
subset $F$ of $\widetilde M$ and an element $x\in \widetilde M$.
The output is whether $\rho(x)\in \mathbb NG\cdot \rho(F)$.
It should be noted that for $G=1$ the subsemimodule membership problem
corresponds to integer programming, which is a classical NP-complete
problem.

Our interest in the subsemimodule membership problem stems from an easy
encoding of it into the submonoid membership problem  for semidirect products.

\begin{lemma}\label{generators}
 Let $G$ be a group with generating set $\Sigma$ and let $M$ be a
$G$-semimodule generated by a subset $B$.  Then the semidirect product
$M\rtimes G$ is generated as a monoid by $\Sigma^{\pm}\cup B$ via the map
$a\mapsto (0,a)$ for $a\in \Sigma^{\pm}$ and $b\mapsto (b,1)$ for $b\in B$.
In particular, if $G$ and $M$ are finitely generated, then so is $M\rtimes G$.
\end{lemma}

\begin{proof}
 As a monoid $M$ is generated by all elements of the form $gb$ with $g\in G$,
$b\in B$.  But $(0,g)(b,1)(0,g^{-1})= (gb,1)$.  It follows that $\Sigma^{\pm}\cup B$
is a monoid generating set for $M\rtimes G$.
\qed
\end{proof}

In light of Lemma~\ref{generators}, we immediately obtain the following result.

\begin{proposition}\label{undecmon}
 Let $G$ be a finitely generated group and $M$ a finitely generated $G$-module
with an undecidable subsemimodule membership problem (for a fixed subsemimodule
$N$).
Then $M\rtimes G$ has an undecidable submonoid membership problem
(for the fixed submonoid $N\rtimes G$).
\end{proposition}

\begin{proof}
The membership of $(m,1)$ in $N\rtimes G$ is evidently equivalent
to the membership of $m\in N$.  Let us just mention how one
effectively transforms input from the subsemimodule problem to
the submonoid membership problem.    Suppose $\Sigma$ is a generating
set for $G$ and $B$ is a generating set for $M$.  Let $\widetilde M$ and
$\rho$ be as before Lemma~\ref{generators}.  Then, for $w\in
(\Sigma^{\pm})^*$, $b\in B$ and $n\in \mathbb Z$, the element
$(\rho(nwb),1)$ is represented in the (group) generating set $\Sigma\cup B$
for $M\rtimes G$ by the word $(wbw^{-1})^n$.  In this way, we
can encode representatives of the module as words in $((\Sigma\cup
B)^{\pm})^*$.
\qed
\end{proof}

If $G$ is a group, the semidirect product $\mathbb ZG\rtimes G$ is the same thing
as the (restricted) wreath product $\mathbb Z\wr G$.  Now if $H$ is a subgroup
of $G$ of index $m$, then it is well known that $\mathbb ZG$ is a free
$\mathbb ZH$-module of rank $m$~\cite{Bro94}.  More precisely, if $T=\{g_1,\ldots,g_m\}$
is a complete set of right coset representatives of $H$ in $G$, then $T$
is a basis for $\mathbb ZG$ as a free left $\mathbb ZH$ module.
Consequently, we have the following lemma.

\begin{lemma}\label{embeds}
Suppose that $H$ is a subgroup of $G$ of index $m$ and $M$ is a
free $\mathbb ZH$-module of rank at most $m$.  Then $M\rtimes H$
embeds as a subgroup of $\mathbb Z\wr G$.
\end{lemma}

\begin{proof}
Clearly $\mathbb ZG\rtimes H\leq \mathbb ZG\rtimes G=\mathbb Z\wr G$.
Since $\mathbb ZG$ is a free $\mathbb ZH$-module of rank $m$, it follows $M\leq
\mathbb ZG$ and so we are done.
\qed
\end{proof}

The main technical result of this paper is the following theorem.

\begin{theorem}\label{mainund}
 There is a free $(\mathbb Z\times \mathbb Z)$-module of finite rank
with an undecidable subsemimodule membership problem for a fixed finitely generated subsemimodule.
\end{theorem}

As a corollary, we obtain that $\mathbb Z\wr (\mathbb Z\times \mathbb Z)$
has an undecidable submonoid membership problem.  This should be contrasted
with the generalized word problem, which is solvable in any finitely generated metabelian
group~\cite{Rom74,Rom80}.  It should be noted that the submodule membership problem
is decidable for free $(\mathbb Z\times \mathbb Z)$-modules~\cite{Sims}, and this is
what underlies the positive solution to the generalized word problem
for metabelian groups in \cite{Rom74,Rom80}.

\begin{corollary}\label{zwrz2undec}
 The submonoid membership problem is undecidable for
$\mathbb Z\wr (\mathbb Z\times \mathbb Z)$ for a fixed finitely generated submonoid.
\end{corollary}

\begin{proof}
By Theorem~\ref{mainund}, there is a free $(\mathbb Z\times \mathbb Z)$-module
$M$ of some rank $m$ with undecidable subsemimodule membership problem
for a fixed subsemimodule.  Proposition~\ref{undecmon} then implies that
$M\rtimes (\mathbb Z\times \mathbb Z)$ has undecidable submonoid membership
for a fixed finitely generated submonoid.  Now $\mathbb Z\times \mathbb Z$
has a subgroup of index $m$ isomorphic to it, e.g.,
$m\mathbb Z\times \mathbb Z$.  Lemma~\ref{embeds} then implies
$M\rtimes (\mathbb Z\times \mathbb Z)$ embeds in
$\mathbb Z\wr (\mathbb Z\times \mathbb Z)$, completing the proof.
\qed
\end{proof}

Recall that a group $G$ is metabelian if it is solvable
of derived length $2$, or equivalently if commutators in $G$
commute. Our next goal is to show that the free metabelian group of
rank $2$ has an undecidable submonoid membership problem for a fixed submonoid.  Since
it is known that free non-cyclic solvable groups of derived length
$3$ or higher have undecidable generalized word problem~\cite{Umi95},
this will show that the submonoid membership problem is
undecidable for free non-abelian solvable groups of any derived length.

We need to recall a description of the free metabelian
group of rank $2$, which is a special case of a more general
result of Almeida~\cite{Alm89}; see also~\cite{MyRoUsVe08}.
In what follows we will work with the {\em Cayley-graph} $\Gamma$ of the group
$\mathbb{Z} \times \mathbb{Z}$.
More precisely, the set of vertices of $\Gamma$ is $\mathbb{Z} \times \mathbb{Z}$ and the
set of (undirected) {\em edges} is
\[\mathscr{E} = \{ \{ (p,q), (r,s) \} \mid p,q,r,s \in \mathbb{Z},
|u-x|+|v-y|=1\}.\]
For $e = \{ (p,q), (r,s) \} \in \mathscr{E}$ and $(a,b) \in \mathbb{Z} \times
\mathbb{Z}$,
we define the translation $e+(a,b) = \{ (p+a,q+b), (r+a,s+b)\} \in \mathscr{E}$.
Let $\Sigma=\{x,y\}$ and label edges in $\Gamma$ of the form
$\{ (p,q), (p+1,q) \}$ (resp. $\{ (p,q), (p,q+1) \}$)
with $x$ (resp. $y$); the reverse edges are labeled with $x^{-1}$
(resp. $y^{-1}$). Let $M_2$ be the free metabelian group generated by $\Sigma$.
Then two words $u,v$ in $(\Sigma^{\pm})^*$ represent the
same element of $M_2$ if and only if they map to the same
element of the free abelian group of rank $2$ and the paths
traversed by $u$ and $v$ in the Cayley graph
$\Gamma$ of $\mathbb Z\times \mathbb Z$ use each edge the
same number of times (where backwards traversals are counted
negatively).  Equivalently, a word $w$ represents the identity
in $M_2$ if and only if it labels a closed path in $\Gamma$ at the origin that
maps to the trivial element of the homology group $H_1(\Gamma)$.  A word $w$ represents an element
of the commutator subgroup $[M_2,M_2]$ if and only if it reads a closed loop in $\Gamma$ at the origin.  Thus
$[M_2,M_2]$ can be identified with
$H_1(\Gamma)$ as a $(\mathbb Z\times \mathbb Z)$-module by
mapping a word $w$ reading a loop at the origin to the element of
$H_1(\Gamma)$ represented by that loop. As a
$(\mathbb Z\times \mathbb Z)$-module, it is free of rank $1$ generated by the commutator
$[x,y]=xyx^{-1}y^{-1}$, which corresponds to
\[c=\{(0,0),(1,0)\} +  \{(1,0),(1,1)\} - \{(1,1),(0,1)\}-\{(0,1),(0,0)\}\]
under our identification of $[M_2,M_2]$ with $H_1(\Gamma)$.
The easiest way to see that $c$ is a free generator is to view $\Gamma$
as the $1$-skeleton of $\mathbb R^2$ with the cell complex structure whose
$2$-cells are the squares of side length $1$ bounded by $\Gamma$.  The fact that
$H_2(\mathbb R^2)=0=H_1(\mathbb R^2)$ says exactly that the boundary
map from the free abelian group on the cells to $H_1(\Gamma)$ (which can be identified with $Z_1(\mathbb R^2)$) is
an isomorphism.  Moreover, the boundary map is actually a
homomorphism of $(\mathbb Z\times \mathbb Z)$-modules since
the action of $\mathbb Z\times \mathbb Z$ on $\mathbb R^2$ is
by cellular maps.  Since $\mathbb Z\times \mathbb Z$ acts
freely and transitively on the cells, it follows that $H_1(\Gamma)$ is freely generated by $c$.

Fix now $m>0$ and consider $H =\langle x^m,y\rangle\leq M_2$.
First note that the image of $H$ in $M_2/[M_2,M_2]=\mathbb Z\times \mathbb Z$
is the subgroup $m\mathbb Z\times \mathbb Z$, which must therefore
be the abelianization of $H$ as it is free of rank $2$ and $H$ is
$2$-generated.
Thus $[H,H]=[M_2,M_2]\cap H$.  Moreover, $[H,H]$ is the $m\mathbb Z\times \mathbb Z$-submodule
of $[M_2,M_2]$ generated by \[c'=\sum_{i=0}^{m-1} (c+(i,0)).\]
Indeed, the elements of $[H,H]$ are the homology classes in $H_1(\Gamma)$ of closed loops in the grid with vertex set $m\mathbb Z\times \mathbb Z$.  If we make $\mathbb R^2$ into
a cell complex by using the squares bounded by this grid,
then the same argument as above shows that $[H,H]$ is freely generated as an $(m\mathbb Z\times \mathbb Z)$-module
by the boundary of the square with vertices \[(0,0),(m,0),(0,1),(m,1).\]  But this is exactly $c'$.

Now as an $(m\mathbb Z\times \mathbb Z)$-module, $[M_2,M_2]$ is free on
$\{c+(i,0)\mid 0\leq i\leq m-1\}$.  But we can then change the basis
to the set \[\{c+(i,0)\mid 0\leq i\leq m-2\}\cup \{c'\}.\] Thus as an
$(m\mathbb Z\times \mathbb Z)$-module $[M_2,M_2]=F\oplus [H,H]$ where
$F$ is free of rank $m-1$.  We can exploit this  to reduce the
subsemimodule membership problem to the submonoid membership problem for $M_2$.

\begin{theorem}\label{freemetabelian}
There is a fixed finitely generated submonoid of the free metabelian
group of rank $2$ with undecidable membership problem.
\end{theorem}
\begin{proof}
By Theorem~\ref{mainund} we can find a free
$(\mathbb Z\times \mathbb Z)$-module $M$ of rank $r$ containing a
fixed finitely generated subsemimodule $N$ with an undecidable
membership problem.  Choose $m=r+1$ and set $H=\langle x^m,y\rangle$.
We saw above that as an
$(m\mathbb Z\times \mathbb Z)$-module we can write $[M_2,M_2]=F\oplus [H,H]$  where $F$
is a free $(m\mathbb Z\times \mathbb Z)$-module of rank $r$.
Since $m\mathbb Z\times \mathbb Z\cong \mathbb Z\times \mathbb Z$,
we can of course find a fixed subsemimodule, which we abusively
denote $N$, inside of $F$ with an undecidable subsemimodule
membership problem.  Consider the submonoid $S$ of $M_2$ generated by
$N$ and $H$.  If $B$ is a finite generating set for $N$, then
$S$ is generated by $B\cup \{x^m,x^{-m},y,y^{-1}\}$ since each translate
of an element of $B$ by an element of $m\mathbb Z\times \mathbb Z$ can be obtained via a conjugation by an element of $H$.   We claim
that $S\cap [M_2,M_2] = N\oplus [H,H]$. Notice that $N\cap [H,H]=0$
since $N\leq F$ and $F\cap [H,H]=0$, so $N+[H,H]=N\oplus [H,H]$.
The inclusion from right to left is trivial.  For the other inclusion,
consider a product $g=h_0n_0\cdots h_kn_k$ with the $h_i\in H$
and the $n_i\in N$ belonging to $[M_2,M_2]$.  Then
\begin{align*}
g &= (h_0n_0h_0^{-1})(h_0h_1n_1(h_0h_1)^{-1})\cdots
(h_0\cdots h_kn_k(h_0\cdots h_k)^{-1})h_0\cdots h_k\\ &= nh_0\cdots h_k
\end{align*}
with $n\in N$.  It follows that $h_0\cdots h_k\in
[M_2,M_2]\cap H=[H,H]$ and so we obtain $g\in N\oplus [H,H]$, as required.

So suppose $x\in F$ and we want to decide whether $x\in N$.
Then since we have $S\cap [M_2,M_2]=N\oplus [H,H]\leq F\oplus [H,H]$,
it follows that $x\in N$ if and only if $x\in S$.  This completes the proof.
\qed
\end{proof}

\section{Tiling problems}
There is a classical connection between Turing machines and
tiling problems \cite{Ber66,Rob71}.  Here we consider
a variant that is most easily translated into the
subsemimodule membership problem.

\begin{figure}[t]
\begin{center}
\begin{picture}(15,15)
 \gasset{Nw=0,Nh=0,Nframe=n,AHnb=0,ELdist=.5}
    \node(tl)(0,15){} \node(tr)(15,15){}
    \node(bl)(0,0){}  \node(br)(15,0){}
    \drawedge[ELside=r](tl,tr){$c_N$}
    \drawedge(tl,bl){$c_W$}
    \drawedge(bl,br){$c_S$}
    \drawedge(br,tr){$c_E$}
  \end{picture}
\end{center}
\caption{The tile $t = (c_N,c_E,c_S,c_W)$}
\label{a tile}
\end{figure}
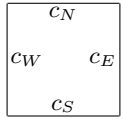

Let $C$ be a finite set of {\em colors} with a
{\em distinguished color} $c_0 \in C$.
A {\em tiling system} over $C$ is
a set $T \subseteq C^4$; its elements are called tiles.
We view a tile $t = (c_N,c_E,c_S,c_W)$
as an edge colored square, as shown in Figure~\ref{a tile}.
We will associate with the tile $t = (c_N,c_E,c_S,c_W)$
the following mapping $\links t \rechts \colon  \mathscr{E} \times C \to \mathbb{Z}$ (where $\mathscr E$ continues to denote the edge set of the Cayley graph of $\mathbb Z\times \mathbb Z$):
\begin{gather*}
\links t \rechts(\{(0,0),(1,0)\}, c_S) = -1 \cdot (1-\delta_{c_S,c_0}) \\
\links t \rechts(\{(1,0),(1,1)\}, c_E) = 1 \cdot (1-\delta_{c_E,c_0}) \\
\links t \rechts(\{(1,1),(0,1)\}, c_N) = 1 \cdot (1-\delta_{c_N,c_0}) \\
\links t \rechts(\{(0,1),(0,0)\}, c_W) = -1 \cdot (1-\delta_{c_W,c_0}) \\
\links t \rechts(e,c) = 0\ \text{in all other cases}
\end{gather*}
where as usual $\delta_{x,y}=1$ when $x=y$ and $\delta_{x,y}=0$ if $x\neq y$.
Thus, we color, for instance, the north edge $\{(0,1),(1,1)\}$
of the cell $\{ (x,y) \mid 0 \leq x,y\leq 1\}$ with the color $c_N$,
in case $c_N  \neq c_0$.
The sign of the value $\links t \rechts(e, c)$ indicates that
the north and east (south and west) edge receive
a positive (negative) orientation. This will be used below, where we
add translates of the maps $\links t \rechts$.  Edges that are colored by $c_0$ receive the value $0$.

Let $f \colon  \mathscr{E} \times C \to \mathbb{Z}$.
We say that $f$ has \emph{finite support} if the set
$f^{-1}(\mathbb{Z}\setminus\{0\})$ is finite.
For $(a,b) \in \mathbb{Z} \times \mathbb{Z}$ we define
the translate $\tau_{a,b}f \colon  \mathscr{E} \times C \to \mathbb{Z}$
as the mapping with
\[\tau_{a,b}f(e,c) = f(e-(a,b),c) \text{ for all } e \in  \mathscr{E}
\text{ and } c \in C.\]
For two mappings $f_1, f_2 \colon   \mathscr{E} \times C
\to \mathbb{Z}$ we define the sum
$f_1 + f_2 \colon  \mathscr{E} \times \Gamma \to \mathbb{Z}$
by
$(f_1+f_2)(e,c) = f_1(e,c)+f_2(e,c)$
for all $e \in \mathscr{E}$ and $c \in C$.
We denote by $0$ the constant
mapping taking the value $0$ everywhere on
$\mathscr{E} \times \Gamma$.
The set of all mappings from $\mathscr{E} \times C$
to $\mathbb{Z}$ forms an abelian group under addition.
The set of all mappings with finite support is a subgroup
of this group.

A {\em tiling sum} over $T$ is a sum of the form
\begin{equation}\label{tilingsum}
f = \sum_{i=1}^n \tau_{x_i,y_i}\links t_i \rechts,
\end{equation}
where $x_i,y_i \in \mathbb{Z}$ and $t_i \in T$ for all $1 \leq i \leq n$.
The evaluation of such a sum yields a
mapping $f \colon  \mathscr{E} \times C \to \mathbb{Z}$.
Note that one may have $(x_i,y_i)=(x_j,y_j)$ for $i \neq j$.
Intuitively one can think of a tiling sum as putting
tiles on certain positions of the grid (one may put
several tiles on the same position or even put the same
tile several times on the same position). When evaluating
the tiling sum, we cancel matching colors on edges, which happens
if, e.g., the color on the north side of a tile matches the
color on the south side of the tile immediately above it.
The distinguished color $c_0$ is not involved in this cancellation
process.  Let us agree to say that the tile $t_i$ is \emph{placed}
in position $(x_i,y_i)$ in the tiling sum \eqref{tilingsum}.
Of course, the same tile may be placed in multiple positions or even multiple times in the same position.

The {\em zero tiling sum problem} for a given tiling system $T$ over
$C$ is the following computational problem:

\medskip

\noindent
INPUT: A mapping
$f_0 \colon  \mathscr{E} \times \Gamma \to \mathbb{Z}$ with finite support.

\medskip

\noindent
QUESTION: Is there a tiling sum $f$ with $f_0+f=0$?

\begin{theorem} \label{thm tiling}
The zero tiling sum problem is undecidable.
\end{theorem}

\begin{proof}
We start with a fixed deterministic Turing machine $M = (Q, \Gamma, \Sigma, \delta, q_0, q_f)$
with an undecidable acceptance problem.
Here, $Q$ is the set of states, $\Gamma$ is the tape alphabet, $\Sigma
\subseteq \Gamma$ is the input alphabet, $q_0 \in Q$ is the initial state,
$q_f \in Q$ is the unique accepting state, and $\delta \colon  Q \times \Gamma \to Q
\times \Gamma \times \{L,R\}$ is the transition mapping ($L$ (resp. $R$) means
that the head moves left (resp. right)).
The blank symbol is $\boxempty \in \Gamma
\setminus \Sigma$. We can make the following assumptions
on the machine $M$:
\begin{itemize}
\item The tape of $M$ is bounded to the left; that is, the machine never moves to the left of the first cell.
\item $M$ terminates if and only if it reaches the accepting state
$q_f$. In particular, an input $w$ is accepted if and only if $M$
terminates on $w$.
\item If $M$ reaches state $q_f$ then the whole tape is blank
and the head of the machine is scanning the left most cell.
\end{itemize}
We take the following fixed set of colors:
\[C = Q \cup \Gamma \cup (Q\times \Gamma)  \cup \{\shortrightarrow, \shortuparrow,   \shortleftarrow,
\shortdownarrow, \begin{turn}{45}\text{$\shortleftarrow$}\end{turn},
\triangleleft, \triangleright, c_0\}.\]
Here, $c_0$ is the distinguished color. In the following pictures
the color $c_0$ will be indicated in a tile by a dotted side.  Also the pair $(q,a)\in Q\times \Gamma$ will be written $qa$.

The set of tiles $T$ consists of the following tiles, which are inspired by
the tiles of the tiling system from \cite[Appendix A]{BoGrGu01}:

\begin{itemize}
\item
Alphabet tiles (for all $a \in \Gamma$):

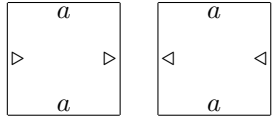
\begin{picture}(60,21)(0,-3)
\put(10,0){
\begin{picture}(15,15)
 \gasset{Nw=0,Nh=0,Nframe=n,AHnb=0,ELdist=.5}
    \node(tl)(0,15){} \node(tr)(15,15){}
    \node(bl)(0,0){}  \node(br)(15,0){}
    \drawedge[ELside=r](tl,tr){$a$}
    \drawedge(tl,bl){$\triangleright$}
    \drawedge(bl,br){$a$}
    \drawedge(br,tr){$\triangleright$}
  \end{picture}
}
\put(30,0){
\begin{picture}(15,15)
 \gasset{Nw=0,Nh=0,Nframe=n,AHnb=0,ELdist=.5}
    \node(tl)(0,15){} \node(tr)(15,15){}
    \node(bl)(0,0){}  \node(br)(15,0){}
    \drawedge[ELside=r](tl,tr){$a$}
    \drawedge(tl,bl){$\triangleleft$}
    \drawedge(bl,br){$a$}
    \drawedge(br,tr){$\triangleleft$}
  \end{picture}
}
\end{picture}
\item
Merging tiles (for all $a \in \Gamma$ and all $p \in Q$):

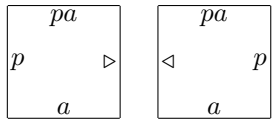
\begin{picture}(60,21)(0,-3)
\put(10,0){
\begin{picture}(15,15)
 \gasset{Nw=0,Nh=0,Nframe=n,AHnb=0,ELdist=.5}
    \node(tl)(0,15){} \node(tr)(15,15){}
    \node(bl)(0,0){}  \node(br)(15,0){}
    \drawedge[ELside=r](tl,tr){$pa$}
    \drawedge(tl,bl){$p$}
    \drawedge(bl,br){$a$}
    \drawedge(br,tr){$\triangleright$}
  \end{picture}
}
\put(30,0){
\begin{picture}(15,15)
 \gasset{Nw=0,Nh=0,Nframe=n,AHnb=0,ELdist=.5}
    \node(tl)(0,15){} \node(tr)(15,15){}
    \node(bl)(0,0){}  \node(br)(15,0){}
    \drawedge[ELside=r](tl,tr){$pa$}
    \drawedge(tl,bl){$\triangleleft$}
    \drawedge(bl,br){$a$}
    \drawedge(br,tr){$p$}
  \end{picture}
}
\end{picture}
\item
Action tiles for moves of the machine $M$:

\begin{picture}(60,29)(0,-3)
\put(10,0){
\begin{picture}(15,20)
 \gasset{Nw=0,Nh=0,Nframe=n,AHnb=0,ELdist=.5}
    \node(tl)(0,15){} \node(tr)(15,15){}
    \node(bl)(0,0){}  \node(br)(15,0){}
    \drawedge[ELside=r](tl,tr){$b$}
    \drawedge(tl,bl){$p$}
    \drawedge(bl,br){$qa$}
    \drawedge(br,tr){$\triangleright$}
    \put(-5,20){if $\delta(q,a) = (p,b,L)$:}
  \end{picture}
}
\put(50,0){
\begin{picture}(15,20)
 \gasset{Nw=0,Nh=0,Nframe=n,AHnb=0,ELdist=.5}
    \node(tl)(0,15){} \node(tr)(15,15){}
    \node(bl)(0,0){}  \node(br)(15,0){}
    \drawedge[ELside=r](tl,tr){$b$}
    \drawedge(tl,bl){$\triangleleft$}
    \drawedge(bl,br){$qa$}
    \drawedge(br,tr){$p$}
    \put(-5,20){if $\delta(q,a) = (p,b,R)$:}
  \end{picture}
}
\end{picture}
\item
Boundary tiles (the labels $b_i$ are just names that we give to these tiles):

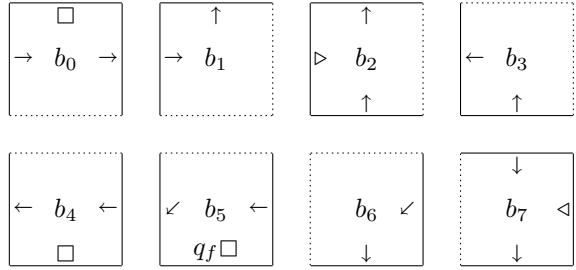
\begin{picture}(120,41)(0,-3)
\put(10,20){
\begin{picture}(15,15)
 \gasset{Nw=0,Nh=0,Nframe=n,AHnb=0,ELdist=.5}
    \node(tl)(0,15){} \node(tr)(15,15){}
    \node(bl)(0,0){}  \node(br)(15,0){}
    \drawedge[ELside=r](tl,tr){$\boxempty$}
    \drawedge(tl,bl){$\shortrightarrow$}
    \drawedge[dash={0.2 0.5}0](bl,br){}
    \drawedge(br,tr){$\shortrightarrow$}
    \node(t0)(7.5,7.5){$b_0$}
  \end{picture}

}
\put(30,20){
\begin{picture}(15,15)
 \gasset{Nw=0,Nh=0,Nframe=n,AHnb=0,ELdist=.5}
    \node(tl)(0,15){} \node(tr)(15,15){}
    \node(bl)(0,0){}  \node(br)(15,0){}
    \drawedge[ELside=r](tl,tr){$\shortuparrow$}
    \drawedge(tl,bl){$\shortrightarrow$}
    \drawedge[dash={0.2 0.5}0](bl,br){}
    \drawedge[dash={0.2 0.5}0](br,tr){}
     \node(t0)(7.5,7.5){$b_1$}
  \end{picture}
}
\put(50,20){
\begin{picture}(15,15)
 \gasset{Nw=0,Nh=0,Nframe=n,AHnb=0,ELdist=.5}
    \node(tl)(0,15){} \node(tr)(15,15){}
    \node(bl)(0,0){}  \node(br)(15,0){}
    \drawedge[ELside=r](tl,tr){$\shortuparrow$}
    \drawedge(tl,bl){$\triangleright$}
    \drawedge(bl,br){$\shortuparrow$}
    \drawedge[dash={0.2 0.5}0](br,tr){}
     \node(t0)(7.5,7.5){$b_2$}
  \end{picture}
}
\put(70,20){
\begin{picture}(15,15)
 \gasset{Nw=0,Nh=0,Nframe=n,AHnb=0,ELdist=.5}
    \node(tl)(0,15){} \node(tr)(15,15){}
    \node(bl)(0,0){}  \node(br)(15,0){}
    \drawedge[dash={0.2 0.5}0](tl,tr){}
    \drawedge(tl,bl){$\shortleftarrow$}
    \drawedge(bl,br){$\shortuparrow$}
    \drawedge[dash={0.2 0.5}0](br,tr){}
    \node(t0)(7.5,7.5){$b_3$}
  \end{picture}
}
\put(10,0){
\begin{picture}(15,15)
 \gasset{Nw=0,Nh=0,Nframe=n,AHnb=0,ELdist=.5}
    \node(tl)(0,15){} \node(tr)(15,15){}
    \node(bl)(0,0){}  \node(br)(15,0){}
    \drawedge[dash={0.2 0.5}0](tl,tr){}
    \drawedge(tl,bl){$\shortleftarrow$}
    \drawedge(bl,br){$\boxempty$}
    \drawedge(br,tr){$\shortleftarrow$}
    \node(t0)(7.5,7.5){$b_4$}
  \end{picture}
}
\put(30,0){
\begin{picture}(15,15)
 \gasset{Nw=0,Nh=0,Nframe=n,AHnb=0,ELdist=.5}
    \node(tl)(0,15){} \node(tr)(15,15){}
    \node(bl)(0,0){}  \node(br)(15,0){}
    \drawedge[dash={0.2 0.5}0](tl,tr){}
    \drawedge(tl,bl){\begin{turn}{45}\text{$\shortleftarrow$}\end{turn}}
    \drawedge(bl,br){${q_f}\boxempty$}
    \drawedge(br,tr){$\shortleftarrow$}
    \node(t0)(7.5,7.5){$b_5$}
  \end{picture}
}
\put(50,0){
\begin{picture}(15,15)
 \gasset{Nw=0,Nh=0,Nframe=n,AHnb=0,ELdist=.5}
    \node(tl)(0,15){} \node(tr)(15,15){}
    \node(bl)(0,0){}  \node(br)(15,0){}
    \drawedge[dash={0.2 0.5}0](tl,tr){}
    \drawedge[dash={0.2 0.5}0](tl,bl){}
    \drawedge(bl,br){$\shortdownarrow$}
    \drawedge(br,tr){\begin{turn}{45}\text{$\shortleftarrow$}\end{turn}}
    \node(t0)(7.5,7.5){$b_6$}
  \end{picture}
}
\put(70,0){
\begin{picture}(15,15)
 \gasset{Nw=0,Nh=0,Nframe=n,AHnb=0,ELdist=.5}
    \node(tl)(0,15){} \node(tr)(15,15){}
    \node(bl)(0,0){}  \node(br)(15,0){}
    \drawedge[ELside=r](tl,tr){$\shortdownarrow$}
    \drawedge[dash={0.2 0.5}0](tl,bl){}
    \drawedge(bl,br){$\shortdownarrow$}
    \drawedge(br,tr){$\triangleleft$}
    \node(t0)(7.5,7.5){$b_7$}
  \end{picture}
}
\end{picture}
\end{itemize}
Now, let $w = w_1 w_2 \cdots w_n$ be an input for the machine $M$
with the $w_i  \in \Sigma$. We associate with $w$ the following
mapping $f_w \colon  \mathscr{E} \times C \to \mathbb{Z}$:
\begin{eqnarray*}
f_w( \{ (0,1),(1,1) \}, \shortdownarrow ) & = & 1 \\
f_w( \{ (1,1),(2,1)\}, q_0 w_1) & = & 1 \\
f_w( \{ (i,1),(i+1,1)\}, w_i) & = & 1 \text{ for }  2 \leq i \leq n \\
f_w( \{ (n+1,1),(n+1,0)\}, \shortrightarrow ) & = & 1.
\end{eqnarray*}
All other values of $f_w$ are $0$, hence $f_w$ has finite support.
As a diagram, the mapping $f_w$ looks as follows:
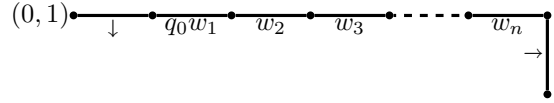
\begin{center}
\setlength{\unitlength}{.7mm}
\begin{picture}(90,21)(0,-3)
 \gasset{Nw=1.5,Nh=1.5,Nframe=n,Nfill=y,AHnb=0,ELdist=.7,linewidth=.6}
    \node[ExtNL=y,NLangle=180](0)(0,15){$(0,1)$}
    \node(1)(15,15){}
    \node(2)(30,15){}
    \node(3)(45,15){}
    \node(4)(60,15){}
    \node(5)(75,15){}
    \node(6)(90,15){}
    \node(a)(90,0){}
    \drawedge[ELside=r](0,1){$\shortdownarrow$}
    \drawedge[ELside=r](1,2){$q_0w_1$}
    \drawedge[ELside=r](2,3){$w_2$}
    \drawedge[ELside=r](3,4){$w_3$}
    \drawedge[dash={1.5}{}](4,5){}
    \drawedge[ELside=r](5,6){$w_n$}
    \drawedge(a,6){$\shortrightarrow$}
  \end{picture}
\end{center}
We will show that $M$ accepts the input $w$ if and only if
there is a tiling sum $f$ with
$f_w + f = 0$.

First assume that there is
such a tiling sum $f$ and let
\begin{equation}\label{tiling sum f}
f = \sum_{i=1}^N \tau_{x_i,y_i}\links t_i \rechts .
\end{equation}


\noindent
{\em Claim 1.} For all $1 \leq i \leq N$, we have both
$x_i,y_i \geq 0$ and either $y_i \geq 1$, or $x_i \geq n+1$ in (\ref{tiling sum f}),
i.e., all tiles are placed into the shaded area in Figure~\ref{shaded area}.

\medskip

\noindent
Let $\preceq$ be the componentwise order on $\mathbb{Z} \times \mathbb{Z}$,
i.e., $(x',y') \preceq (x,y)$ if and only if $x' \leq x$ and $y'\leq y$.
In order to deduce a contradiction, assume that there exists a tile of $f$
placed outside the shaded area and suppose that $i$ is chosen so that
$(x_i,y_i)$ is $\preceq$-minimal with $\tau_{(x_i,y_i)}\links t_i\rechts$ outside of the shaded area.
\begin{figure}[t]
\begin{center}
\setlength{\unitlength}{.5mm}
\begin{picture}(130,90)(0,-3)
    \gasset{Nw=2,Nh=2,Nframe=n,Nfill=y,AHnb=0,ELdist=.9,linewidth=.5}
   \drawpolygon[linewidth=0,fillgray=0.75](0,15)(90,15)(90,0)(130,0)(130,80)(0,80)
    \gasset{linewidth=.7}
    \node[ExtNL=y,NLangle=180](0)(0,15){$(0,1)$}
    \node(1)(15,15){}
    \node(2)(30,15){}
    \node(3)(45,15){}
    \node(4)(60,15){}
    \node(5)(75,15){}
   \node[ExtNL=y,NLangle=0](6)(90,15){$(n+1,1)$}
   \node(a)(90,0){}
    \drawedge[ELside=r](0,1){$\shortdownarrow$}
    \drawedge[ELside=r](1,2){$q_0w_1$}
    \drawedge[ELside=r](2,3){$w_2$}
    \drawedge[ELside=r](3,4){$w_3$}
    \drawedge[dash={1.5}{}](4,5){}
    \drawedge[ELside=r](5,6){$w_n$}
    \drawedge(a,6){$\shortrightarrow$}
    \put(131,40){$\ldots$}
    \put(65,81){$\vdots$}
  \end{picture}
\end{center}
\caption{\label{shaded area}}
\end{figure}
Note that for every tile in $T$, the south or the west edge is
colored differently from $c_0$.
Hence, the south or the west color of tile $t_i$ is different
from $c_0$. In order to match this up, there must exist a tile placed
to the south or to the west of $t_i$, that is, there must be $1 \leq j
\leq N$  such that either $x_j = x_i$ and $y_j = y_i-1$, or $x_j = x_i-1$ and $y_j = y_i$.
This contradicts the choice of $i$.

Recall that $b_0$ and $b_1$ are two boundary tiles.

\medskip

\noindent
{\em Claim 2.}  There exists $m \geq n+1$ such that
the tiling sum $f$ in (\ref{tiling sum f}) can be written as
\begin{equation} \label{g1}
f =  \sum_{i=n+1}^{m-1} \tau_{i,0}\links b_0 \rechts +
\tau_{m,0}\links b_1 \rechts + g_1,
\end{equation}
where $g_1$ is a tiling sum, which does not contain a
summand of the form $\tau_{x,0}\links t \rechts$ for some
$x \in \mathbb{Z}$ and some tile $t$.

\medskip

\noindent
Since $f_w( \{ (n+1,1),(n+1,0)\}, \shortrightarrow ) = 1$,
$f$ must contain a summand of the form $\tau_{x,0} \links t \rechts$ with $x \geq n+1$.
Let $m$ be the maximal $x$ with this property.
The tile $t$ must be $b_1$, because every other tile has a color different
from $c_0$ on its east side or on its south side. Then, $f$ would contain
a summand of the form $\tau_{m+1,0} \links t' \rechts$ (which contradicts the choice of
$m$) or $\tau_{m,-1} \links t' \rechts$ (which contradicts Claim 1). Hence,
$f$ contains the summand $\tau_{m,0} \links b_1 \rechts$.
Now, by induction on $i$ we can easily show that $f$ must contain
all summands $\tau_{i,0} \links b_0 \rechts$ for $n+1 \leq i \leq m-1$.
For this, note that $b_0$ is the only tile with color $\shortrightarrow$
on its east side. Hence, we can write $f$ as
$f = \sum_{i=n+1}^{m-1} \tau_{i,0} \links b_0 \rechts + \tau_{m,0}\links b_1 \rechts + g_1$
for some tiling sum $g_1$. The diagram of the evaluation
of the sum
\[f_1 = f_w + \sum_{i=n+1}^{m-1} \tau_{i,0}\links b_0 \rechts + \tau_{m,0}\links b_1 \rechts \]
is shown in Figure~\ref{fig initial config}.
Note that we have
\[f_w + f = f_1 + g_1 = 0 .\]
Now, assume that $g_1$ contains a summand of the form
$\tau_{x,0}\links t \rechts$
for  some $x$ and some tile $t$. Choose $x$ minimal with
this property. Since the west or the south side
of tile $t$ has a color different from $c_0$, the sum
$g_1$ must contain a summand of the form $\tau_{x-1,0}\links t' \rechts$
(which contradicts the choice of $x$) or of the form
$\tau_{x,-1}\links t' \rechts$ (which contradicts Claim~1).
This proves Claim~2.
\begin{figure}[t]
\begin{center}
\setlength{\unitlength}{.6mm}
\begin{picture}(160,10)(0,10)
 \gasset{Nw=1.5,Nh=1.5,Nframe=n,Nfill=y,AHnb=0,ELdist=.7,linewidth=.6}
    \node[ExtNL=y,NLangle=180](0)(0,15){$(0,1)$}
    \node(1)(15,15){}
    \node(2)(30,15){}
    \node(3)(45,15){}
    \node(4)(60,15){}
    \node(5)(75,15){}
    \node(6)(90,15){}
    \node(7)(105,15){}
    \node(8)(120,15){}
   \node(9)(135,15){}
   \node[ExtNL=y,NLangle=0](10)(150,15){$(m+1,1)$}
   \drawedge[ELside=r](0,1){$\shortdownarrow$}
    \drawedge[ELside=r](1,2){$q_0w_1$}
    \drawedge[ELside=r](2,3){$w_2$}
    \drawedge[ELside=r](3,4){$w_3$}
    \drawedge[dash={1.5}{}](4,5){}
    \drawedge[ELside=r](5,6){$w_n$}
    \drawedge[ELside=r](6,7){$\boxempty$}
    \drawedge[dash={1.5}{}](7,8){}
    \drawedge[ELside=r](8,9){$\boxempty$}
    \drawedge[ELside=r](9,10){$\shortuparrow$}
  \end{picture}
\end{center}
\caption{\label{fig initial config} The evaluation of the sum
$\displaystyle{f_1 = f_w + \sum_{i=n+1}^{m-1} \tau_{i,0}\links b_0 \rechts +\tau_{m,0}\links b_1 \rechts}$}
\end{figure}

\medskip

\noindent
By Claims~1 and 2, we know that all summands in $g_1$ are of the form
$\tau_{x,y}\links t \rechts$ with $x \geq 0$ and $y \geq 1$.
Moreover, $g_1$ added to $f_1$ in Figure~\ref{fig initial config} gives $0$.

\medskip

\medskip

\noindent
{\em Claim 3.}
The tiling sum $g_1$ does not contain a summand of the form
$\tau_{x,y}\links t \rechts$ ($x \geq 0,y \geq 1$) with $t \in \{b_0,b_1\}$.

\medskip

\noindent
Assume that $g_1$ contains the summand $\tau_{x,y}\links t \rechts$
($x \geq 0,y \geq 1$) with $t \in \{b_0,b_1\}$
and assume that $x$ is minimal with this property.
Since the west edge of $t$ is labeled with $\shortrightarrow$,
$g_1$ has to contain the summand $\tau_{x-1,y}\links b_0 \rechts$, which
is again a contradiction.

\medskip

\noindent
{\em Claim 4.}
The tiling sum $g_1$ does not contain a summand of the form
$\tau_{x,y}\links t \rechts$ with $x \geq m+1$.

\medskip

\noindent
Assume that $g_1$ contains the summand $\tau_{x,y} \links t \rechts$ with $x \geq m+1$
and assume that $y$ is minimal with this property.
Since $t \not\in \{b_0,b_1\}$ by Claim 3, the south side of $t$
is labeled with a color different from $c_0$. In order to match
this up, $g_1$ has to contain also a summand of the from
$\tau_{x,y-1} \links t \rechts$, which contradicts the minimality of $y$.

\medskip

\noindent
{\em Claim 5.}
For every position $(x,y)$
with $0 \leq x \leq m, y \geq 1$, the tiling sum $f$ does not
contain a summand of the form $\tau_{x,y}\links t \rechts + \tau_{x,y}\links t' \rechts$
(possibly with $t=t'$),  i.e., no position $(x,y)$ with
$0 \leq x \leq m, y \geq 1$ receives two tiles.

\medskip

\noindent
Assume that $g_1 = \tau_{x,y}\links t \rechts + \tau_{x,y}\links t' \rechts + g'_1$ ($0 \leq x \leq m$, $y\geq 1$),
i.e., position $(x,y)$ receives at least two tiles.
We can assume that $y$ is minimal with this property.
Since $\{t,t'\} \cap \{b_0,b_1\} = \emptyset$ by Claim 3, the south side
of $t$ ($t'$, resp.) is labeled with a color $c \neq c_0$
($c' \neq c_0$, resp.). Hence, the edge $\{ (x,y), (x+1,y) \}$ receives the
colors $c$ and $c'$ (we may have $c=c'$, i.e.,
$\{ (x,y), (x+1,y) \}$ receives the color $c$ twice).
If $y \geq 2$, then we have to match this up by putting
at least two tiles on position $(x,y-1)$. Since this
contradicts the choice of $y$, we may assume that $y=1$.
Recall that $0 \leq x \leq m$.
Let $u_0$ and $u_1$ be two tiles that are put onto position
$(x,1)$ (we may have $u_0 = u_1$).
Since the south edges of $u_0$ and $u_1$ are labeled
with colors different from $c_0$ and since
the edge $\{(x,1), (x+1,1) \}$ in Figure~\ref{fig initial config}
is labeled with a single color exactly once, $g_1$ has
to contain a summand of the form $\tau_{x,0} f_u$ for
some tile $u$. This contradicts Claim~2 and proves
Claim~5.

\medskip

\noindent
Now that we have established Claims~1--5, we  are essentially faced with a classical
tiling problem. We have to find a tiling (in the classical sense, where
each grid point gets at most one tile), such that the south side
of the final tiling is labeled with the line in
Figure~\ref{fig initial config} and all other boundary edges are labeled
with the distinguished color $c_0$. Note that the line in
Figure~\ref{fig initial config} is labeled with the word
$C_1 = \,  \shortdownarrow \!\!  q_0w_1 \, w_2 \cdots w_n \boxempty^{m-n-1}\!\!
\shortuparrow$, which represents the initial configuration
for the input $w = w_1 w_2 \cdots w_n$.
Recall that we want to show that $M$ finally accepts the input $w$,
which is equivalent to the fact that $M$ finally terminates on
input $w$.
In order to deduce a contradiction, assume that
$M$ does not terminate on input $w$. Let $C_i$ ($i \geq 1$) be the unique
configuration that is reached from $C_1$ after $i-1$ steps.
We can view every $C_i$ as a word over the alphabet
$\Gamma \cup (Q\times \Gamma) \cup \{
\shortuparrow, \shortdownarrow\}$ starting (ending) with
$\shortdownarrow$ ($\shortuparrow$).  Here $\shortdownarrow$ ($\shortuparrow$) marks the beginning (end) of the tape. The fact that the machine is in state $q$ with the tape head over the symbol $a$ is indicated by an occurrence of $(q,a)$ (which we abbreviate to $qa$) in $C_i$.
By padding words with blanks, we can assume for every
$i \geq 1$: either $|C_i|=m+1$ (this is the case for
$C_1$) or $|C_i|>m+1$ and $C_i$ does not end with $\boxempty\!\shortuparrow$
(which means that $C_i$ cannot be represented by a shorter
configuration word).

We prove by induction on $i$ that, for every $i \geq 1$, we have $|C_i|=m+1$ and
the tiling sum $g_1$ in (\ref{g1}) can be written
as $g_1 = h_i + g_i$ where all summands in $h_i$ ($g_i$, resp.)
are of the form $\tau_{x,y} \links t \rechts$ with $0 \leq x \leq m$ and $1 \leq y \leq  i-1$
($0 \leq x \leq m$ and $y \geq  i$, resp.) and
the diagram of the evaluation of
$f_1 + h_i$ is as shown in Figure~\ref{fig C_i},
\begin{figure}[t]
\begin{center}
\setlength{\unitlength}{.7mm}
\begin{picture}(145,10)(0,10)
 \gasset{Nw=1.5,Nh=1.5,Nframe=n,Nfill=y,AHnb=0,ELdist=.7,linewidth=.6}
    \node[ExtNL=y,NLangle=180](0)(0,15){$(0,i)$}
    \node[ExtNL=y,NLangle=0](9)(135,15){$(m+1,i)$}
    \node(1)(15,15){}
    \node(2)(30,15){}
    \node(3)(45,15){}
    \node(4)(60,15){}
    \node(5)(75,15){}
    \node(6)(90,15){}
    \node(7)(105,15){}
    \node(8)(120,15){}
    \drawedge[ELside=r](0,1){$\shortdownarrow$}
    \drawedge[ELside=r](1,2){$u_1$}
    \drawedge[dash={1.5}{}](2,3){}
    \drawedge[ELside=r](3,4){$u_{j-1}$}
    \drawedge[ELside=r](4,5){$q u_j$}
    \drawedge[ELside=r](5,6){$u_{j+1}$}
    \drawedge[dash={1.5}{}](6,7){}
    \drawedge[ELside=r](7,8){$u_{m-1}$}
    \drawedge[ELside=r](8,9){$\shortuparrow$}
  \end{picture}
\end{center}
\caption{\label{fig C_i} The evaluation of the sum $f_1+h_i$}
\end{figure}
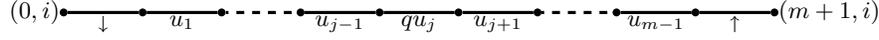
where $C_i = \, \shortdownarrow \!\! u_1 u_2 \cdots u_{j-1} qu_j u_{j+1} \cdots
u_{m-1} \!\!  \shortuparrow$ with the $u_i\in \Gamma$.  This will contradict the fact that $g_1$ is a finite
tiling sum.  It will therefore follow that $w$ is accepted by the machine.

For $i=1$ we  take $h_1=0$. Assume that the above statement is already shown
for $i \geq 1$. We have $0 = f_1+g_1 = (f_1+ h_i)+g_i$, where the evaluation of the sum
$f_1+h_i$ represents the configuration $C_i$ as shown in Figure~\ref{fig C_i}.
All summands in $g_i$ are of the form $\tau_{x,y} \links t \rechts$ with $0 \leq x \leq m$ and $y \geq  i$.

Note that it is not possible
that $j=1$ and $\delta(q,u_j) \in Q \times \Gamma \times \{L\}$
(the machine $M$ is programmed in such way that it does not
cross the left end of the tape). Moreover, $q$ is not the final state $q_f$ since we are assuming that $w$ is not accepted.
We distinguish two cases. Suppose first
$j=m-1$ and $\delta(q,u_j) \in Q \times \Gamma \times \{R\}$.  Then  the diagram of the evaluation of
$f_1 + h_i$ in fact has the following shape:
\begin{center}
\setlength{\unitlength}{.8mm}
\begin{picture}(115,10)(0,10)
 \gasset{Nw=1.5,Nh=1.5,Nframe=n,Nfill=y,AHnb=0,ELdist=.7,linewidth=.6}
    \node[ExtNL=y,NLangle=180](0)(0,15){$(0,i)$}
    \node[ExtNL=y,NLangle=0](7)(105,15){$(m+1,i)$}
    \node(1)(15,15){}
    \node(2)(30,15){}
    \node(3)(45,15){}
    \node(4)(60,15){}
    \node(5)(75,15){}
    \node(6)(90,15){}
    \drawedge[ELside=r](0,1){$\shortdownarrow$}
    \drawedge[ELside=r](1,2){$u_1$}
    \drawedge[ELside=r](2,3){$u_2$}
     \drawedge[dash={1.5}{}](3,4){}
    \drawedge[ELside=r](4,5){$u_{m-2}$}
    \drawedge[ELside=r](5,6){$q u_{m-1}$}
    \drawedge[ELside=r](6,7){$\shortuparrow$}
  \end{picture}
\end{center}
The only possible tiles that can be placed in position $(m-1,i)$ are action
tiles with south side colored $qu_{m-1}$.  Since the machine is
deterministic and $\delta(q,u_{m-1}) \in Q \times \Gamma \times \{R\}$, the unique
such action tile has the shape
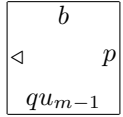
\begin{center}
\setlength{\unitlength}{1mm}
\begin{picture}(15,15)
 \gasset{Nw=0,Nh=0,Nframe=n,AHnb=0,ELdist=.5}
    \node(tl)(0,15){} \node(tr)(15,15){}
    \node(bl)(0,0){}  \node(br)(15,0){}
    \drawedge[ELside=r](tl,tr){$b$}
    \drawedge(tl,bl){$\triangleleft$}
    \drawedge(bl,br){$qu_{m-1}$}
    \drawedge(br,tr){$p$}
  \end{picture}
  \end{center}
and so this tile must be placed in position $(m-1,i)$.
But since there is no tile with west side $p \in Q$ and south side
$\shortuparrow$, we obtain a contradiction thanks to Claim 5.

Next suppose that either $j<m-1$, or $j=m-1$ and $\delta(q,u_j) \in Q \times \Gamma \times \{L\}$.  Then certainly $|C_{i+1}|=m+1$.
Now, we can match up the edges in Figure~\ref{fig C_i}
in exactly one way: In position $(j,i)$ we have to put the unique action tile
with south side $qu_j$ (this tile is unique, since $M$ is deterministic).
Depending on whether  $\delta(q,u_j) \in Q \times \Gamma \times \{L\}$
or $\delta(q,u_j) \in Q \times \Gamma \times \{R\}$, we have to put
one of the two merging tiles either to the left or to the right of
the action tile. The rest of the row is filled up with alphabet tiles
and the boundary tile $b_7$ ($b_2$, resp.) at position $(0,i)$
($(m,i)$, resp.) (using that only these types of tiles have
$\triangleleft$ on their east side or $\triangleright$ on their west side).
The claims ensure no further tiles may be placed. In case $\delta(q,u_j)=(p,b,L)$, the tiling looks as
in Figure~\ref{machinecomputation}.
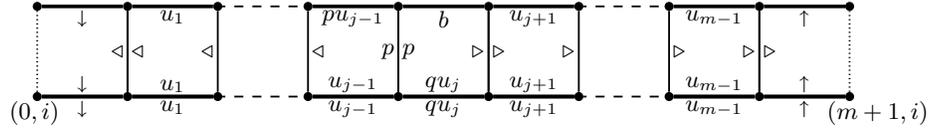
\begin{figure}[h!t!b!p!]
\begin{center}
\setlength{\unitlength}{.8mm}
\begin{picture}(145,25)(-1,10)
 \gasset{Nw=1.5,Nh=1.5,Nframe=n,Nfill=y,AHnb=0,ELdist=.7,linewidth=.6}
    \node[ExtNL=y,NLangle=-100](0)(0,15){$(0,i)$}
    \node(1)(15,15){}
    \node(2)(30,15){}
    \node(3)(45,15){}
    \node(4)(60,15){}
    \node(5)(75,15){}
    \node(6)(90,15){}
    \node(7)(105,15){}
    \node(8)(120,15){}
    \node[ExtNL=y,NLangle=-30](9)(135,15){$(m+1,i)$}
    \node(0')(0,30){}
    \node(1')(15,30){}
    \node(2')(30,30){}
    \node(3')(45,30){}
    \node(4')(60,30){}
    \node(5')(75,30){}
    \node(6')(90,30){}
    \node(7')(105,30){}
    \node(8')(120,30){}
    \node(9')(135,30){}
    \drawedge[ELside=r](0,1){$\shortdownarrow$}
    \drawedge[ELside=r](1,2){$u_1$}
    \drawedge[ELside=r](3,4){$u_{j-1}$}
    \drawedge[ELside=r](4,5){$q u_j$}
    \drawedge[ELside=r](5,6){$u_{j+1}$}
    \drawedge[ELside=r](7,8){$u_{m-1}$}
    \drawedge[ELside=r](8,9){$\shortuparrow$}
    \drawedge(0,1){$\shortdownarrow$}
    \drawedge(1,2){$u_1$}
    \drawedge(3,4){$u_{j-1}$}
    \drawedge(4,5){$q u_j$}
    \drawedge(5,6){$u_{j+1}$}
    \drawedge(7,8){$u_{m-1}$}
    \drawedge(8,9){$\shortuparrow$}
    \drawedge[ELside=r](0',1'){$\shortdownarrow$}
    \drawedge[ELside=r](1',2'){$u_1$}
    \drawedge[ELside=r](3',4'){$pu_{j-1}$}
    \drawedge[ELside=r](4',5'){$b$}
    \drawedge[ELside=r](5',6'){$u_{j+1}$}
    \drawedge[ELside=r](7',8'){$u_{m-1}$}
    \drawedge[ELside=r](8',9'){$\shortuparrow$}
    \gasset{linewidth=.3}
    \drawedge[dash={1.5}{}](2,3){}
    \drawedge[dash={1.5}{}](2',3'){}
     \drawedge[dash={1.5}{}](6,7){}
    \drawedge[dash={1.5}{}](6',7'){}
    \drawedge[dash={0.2 0.5}0](0,0'){}
    \drawedge(1,1'){$\triangleleft$}\drawedge[ELside=r](1,1'){$\triangleleft$}
    \drawedge(2,2'){$\triangleleft$}
    \drawedge[ELside=r](3,3'){$\triangleleft$}
    \drawedge(4,4'){$p$}\drawedge[ELside=r](4,4'){$p$}
    \drawedge(5,5'){$\triangleright$}\drawedge[ELside=r](5,5'){$\triangleright$}
    \drawedge(6,6'){$\triangleright$}
    \drawedge[ELside=r](7,7'){$\triangleright$}
    \drawedge(8,8'){$\triangleright$}\drawedge[ELside=r](8,8'){$\triangleright$}
    \drawedge[dash={0.2 0.5}0](9,9'){}
  \end{picture}
\end{center}
\caption{Simulating a move of the Turing machine\label{machinecomputation}}
\end{figure}
We define $h_{i+1}$ as the sum of $h_i$ and all summands
$\tau_{x,i} \links t(x,i) \rechts$, where $0 \leq x \leq m$ and $t(x,i)$ is the unique tile
that we  put on position $(x,i)$. The tiling sum $g_{i+1}$ is $g_i$ without
these summands $\tau_{x,i} \links t(x,i) \rechts$.
We now have shown that $w$ is accepted by $M$ if there exists a tiling sum $f$
with $f_w+f=0$.

For the other direction, we assume that $w$ is accepted by $M$.
We have to show that there exists a tiling sum $f$ with $f_w+f=0$.
This is much easier than what we have already done.  Since $w$ is accepted
by $M$, there exists a number $m-1$ (the space consumption of $M$ on input
$w$) and sequence of configurations (encoded as before) $C_1, C_2, \ldots, C_N$ (all of length $m+1$)
such that $C_1 = \shortdownarrow \! q_0w \boxempty^{m-n-1} \! \shortuparrow$ is the initial configuration for the input $w$,
$C_N$ is of the form $\shortdownarrow \! q_f \boxempty^{m-1}\! \shortuparrow$, and $M$ moves from $C_i$ to $C_{i+1}$
in one step ($1 \leq i \leq N-1$).
{}From this computation we can build up a tiling in the standard way
(every position receives at most one tile) to obtain the tiling sum $f$,
essentially by reversing the previous argument. Namely, we first add to
$f_w$ the sum
$\sum_{i=n+1}^{m-1}\tau_{i,0}\links b_0\rechts+\tau_{m,0}\links b_1\rechts$
to obtain Figure~\ref{fig initial config}.  Then one continues as per
Figure~\ref{machinecomputation} to build up rows $2$ through $N$.
In this way one obtains a sum $f_w+g$, with $g$ a tiling sum, whose evaluation looks like:
\begin{center}
\setlength{\unitlength}{.8mm}
\begin{picture}(115,10)(0,10)
 \gasset{Nw=1.5,Nh=1.5,Nframe=n,Nfill=y,AHnb=0,ELdist=.7,linewidth=.6}
    \node[ExtNL=y,NLangle=180](0)(0,15){$(0,N)$}
    \node[ExtNL=y,NLangle=0](7)(105,15){$(m+1,N)$}
    \node(1)(15,15){}
    \node(2)(30,15){}
    \node(3)(45,15){}
    \node(4)(60,15){}
    \node(5)(75,15){}
    \node(6)(90,15){}
    \drawedge[ELside=r](0,1){$\shortdownarrow$}
    \drawedge[ELside=r](1,2){$q_f\boxempty$}
    \drawedge[ELside=r](2,3){$\boxempty$}
     \drawedge[dash={1.5}{}](3,4){}
    \drawedge[ELside=r](4,5){$\boxempty$}
    \drawedge[ELside=r](5,6){$\boxempty$}
    \drawedge[ELside=r](6,7){$\shortuparrow$}
  \end{picture}
\end{center}
Finally, we complete the tiling as follows.
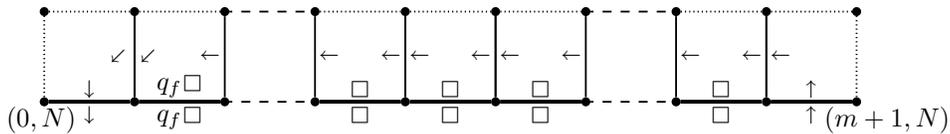
\begin{center}
\setlength{\unitlength}{.8mm}
\begin{picture}(145,25)(-1,10)
 \gasset{Nw=1.5,Nh=1.5,Nframe=n,Nfill=y,AHnb=0,ELdist=.7,linewidth=.6}
    \node[ExtNL=y,NLangle=-100](0)(0,15){$(0,N)$}
    \node(1)(15,15){}
    \node(2)(30,15){}
    \node(3)(45,15){}
    \node(4)(60,15){}
    \node(5)(75,15){}
    \node(6)(90,15){}
    \node(7)(105,15){}
    \node(8)(120,15){}
    \node[ExtNL=y,NLangle=-30](9)(135,15){$(m+1,N)$}
    \node(0')(0,30){}
    \node(1')(15,30){}
    \node(2')(30,30){}
    \node(3')(45,30){}
    \node(4')(60,30){}
    \node(5')(75,30){}
    \node(6')(90,30){}
    \node(7')(105,30){}
    \node(8')(120,30){}
    \node(9')(135,30){}
    \drawedge[ELside=r](0,1){$\shortdownarrow$}
    \drawedge[ELside=r](1,2){$q_f\boxempty$}
    \drawedge[ELside=r](3,4){$\boxempty$}
    \drawedge[ELside=r](4,5){$\boxempty$}
    \drawedge[ELside=r](5,6){$\boxempty$}
    \drawedge[ELside=r](7,8){$\boxempty$}
    \drawedge[ELside=r](8,9){$\shortuparrow$}
    \drawedge(0,1){$\shortdownarrow$}
    \drawedge(1,2){$q_f\boxempty$}
    \drawedge(3,4){$\boxempty$}
    \drawedge(4,5){$\boxempty$}
    \drawedge(5,6){$\boxempty$}
    \drawedge(7,8){$\boxempty$}
    \drawedge(8,9){$\shortuparrow$}
    \gasset{linewidth=.3}
    \drawedge[dash={0.2 0.5}0](0',1'){}
    \drawedge[dash={0.2 0.5}0](1',2'){}
    \drawedge[dash={0.2 0.5}0](3',4'){}
    \drawedge[dash={0.2 0.5}0](4',5'){}
    \drawedge[dash={0.2 0.5}0](5',6'){}
    \drawedge[dash={0.2 0.5}0](7',8'){}
    \drawedge[dash={0.2 0.5}0](8',9'){}
    \drawedge[dash={1.5}{}](2,3){}
    \drawedge[dash={1.5}{}](2',3'){}
     \drawedge[dash={1.5}{}](6,7){}
    \drawedge[dash={1.5}{}](6',7'){}
    \drawedge[dash={0.2 0.5}0](0,0'){}
    \drawedge(1,1'){\begin{turn}{45}\text{$\shortleftarrow$}\end{turn}}
    \drawedge[ELside=r](1,1'){\begin{turn}{45}\text{$\shortleftarrow$}\end{turn}}
    \drawedge(2,2'){$\shortleftarrow$}
    \drawedge[ELside=r](3,3'){$\shortleftarrow$}
    \drawedge(4,4'){$\shortleftarrow$}\drawedge[ELside=r](4,4'){$\shortleftarrow$}
    \drawedge(5,5'){$\shortleftarrow$}\drawedge[ELside=r](5,5'){$\shortleftarrow$}
    \drawedge(6,6'){$\shortleftarrow$}
    \drawedge[ELside=r](7,7'){$\shortleftarrow$}
    \drawedge(8,8'){$\shortleftarrow$}\drawedge[ELside=r](8,8'){$\shortleftarrow$}
    \drawedge[dash={0.2 0.5}0](9,9'){}
  \end{picture}
\end{center}
Formally, $f=g+\tau_{(0,N)}\links b_6\rechts+\tau_{(1,N)}\links b_5\rechts+
\sum_{i=2}^{m-1}\tau_{(i,N)}\links b_4\rechts+\tau_{(m,N)}\links b_3\rechts$ is a tiling sum with $f_w+f=0$.
This completes the proof of Theorem~\ref{thm tiling}.
\qed
\end{proof}

We now proceed to the proof of Theorem~\ref{mainund}, thereby establishing
Corollary~\ref{zwrz2undec} and Theorem~\ref{freemetabelian}.  We recall that
if $G$ is a group, then the free $G$-module on a set $X$ can be realized as
the abelian group of all finitely supported functions $f\colon G\times X\to
\mathbb Z$ with pointwise addition and module action
given by $g_0f(g,x) = f(g_0^{-1} g,x)$.

\begin{proof}[Theorem~\ref{mainund}]
The abelian group $M$ of all finitely supported functions from
$\mathscr E\times C$ to $\mathbb Z$ is a free $\mathbb Z\times \mathbb Z$
module of rank $2|C|$ via the translation action.  Indeed, let
us set $r=\{(0,0),(1,0)\}$ and $u=\{(0,0),(0,1)\}$.  Then
$\mathscr E\times C= (\mathbb Z\times \mathbb Z)\times \{r,u\}\times C$
since each horizontal edge is uniquely of the form $(a,b)+r$ and
each vertical edge is uniquely of the form $(a,b)+u$.  A tiling
sum is precisely an element of the subsemimodule $N$ of $M$
generated by the set $\{\links t \rechts\mid t\in T\}$.
Then the zero tiling sum problem is asking exactly whether there
exists $f\in N$ so that $f_0+f=0$, which is equivalent to asking
whether $-f_0\in N$.   Theorem~\ref{thm tiling} provides a fixed tiling system
with undecidable zero tiling sum problem.  Therefore this is a fixed finitely
generated subsemimodule of a fixed free
$(\mathbb Z\times \mathbb Z)$-module with undecidable subsemimodule membership problem.
This completes the proof.
\qed
\end{proof}

\section{Rational subsets of two-dimensional lamplighter groups}

By a \emph{two-dimensional lamplighter group}, we mean a wreath product of the
form $\mathbb Z/n\mathbb Z\wr (\mathbb Z\times \mathbb Z)$ for $n \geq 2$.  In this section,
we show that the rational subset membership problem is undecidable for such
groups.  By an \emph{effective ring}, we mean a unital ring $R$ whose arithmetic can
be represented effectively (like $\mathbb Z$ or $\mathbb Z/n\mathbb Z$).  Let $G$ be a
finitely generated group with generating set $\Sigma$.  The \emph{subset sum
  problem} for a finitely generated $RG$-module $M$ is the following
algorithmic problem.  Given $m\in M$ and a finite subset $F\subseteq M$ of non-zero elements,
determine whether there exist distinct elements $g_1,\ldots, g_n\in G$ and
elements $f_1,\ldots, f_n\in F$ (not necessarily distinct) so that
$m=\sum_{i=1}^n g_if_i$. In the case the answer is ``yes'',
we say that $m$ is a \emph{subset sum} of $F$.   If $F$ is fixed, then we call this the subset sum problem for $F$.

\begin{theorem}\label{subsetsumundec}
Let $R \neq 0$ be an effective ring.  Then there is a free
$R(\mathbb Z\times \mathbb Z)$-module of finite rank
and a fixed finite subset $F$ of non-zero elements so that the subset sum problem for $F$ is undecidable.
\end{theorem}

\begin{proof}
Let $T$ be the fixed tiling system with undecidable zero tiling sum problem
constructed earlier.  We now consider mappings $f\colon \mathscr E\times
C\to R$ instead of mappings to $\mathbb Z$, but otherwise retain the
definitions and notation from the proof of Theorem~\ref{thm tiling}.  The
proof of that theorem shows that $w$ is accepted by the Turing machine if and
only if $-f_w$ is a subset sum of $F=\{\links t\rechts\mid t\in T\}$.  Indeed,
the proof shows that if $w$ is accepted by the Turing machine, then there is a
tiling sum $f$ in which no two tiles are placed in the same position and so
that $f_w+f=0$.  Conversely, if $-f_w$ is a subset sum of $F$, then we can write
$f_w+f=0$ with $f$ a tiling sum never placing two tiles in the same
position. The argument of Theorem~\ref{thm tiling} now shows that
$w$ must be accepted by the Turing machine, the only difference being
that Claim 5 is now an assumption rather than a result that must be proved.
\qed
\end{proof}

We now aim to show that two-dimensional lamplighter groups have undecidable rational subset membership problem.

\begin{proposition}\label{prop undecrat}
Let $R=\mathbb Z/n\mathbb Z$ ($n \geq 2$) and suppose that $M$ is a finite rank free
$R[\mathbb Z\times \mathbb Z]$-module with fixed finite subset $F$ of non-zero elements having an
undecidable subset sum problem.  Then there is a fixed rational
subset of $M\rtimes (\mathbb Z\times\mathbb Z)$ with undecidable membership problem.
\end{proposition}

\begin{proof}
Let $B$ be a basis for $M$ and take as a generating set for $G=M\rtimes
(\mathbb Z\times\mathbb Z)$ the set $B\cup \{x,y\}$ where $x=(1,0)$ and
$y=(0,1)$.  We claim that $m\in M$ is a subset sum of $F$ if and
only if $(m,(0,0))$ belongs to the rational subset
\[L=\{x^{\pm 1},y^{\pm 1}\}^*[(x\cup Fx)^*y(x^{-1})^*]^*\{x^{\pm 1},y^{\pm 1}\}^*.\]
Let us give a high level description of how this works.  The first term
$\{x^{\pm 1},y^{\pm 1}\}^*$ in $L$ lets us move to any position in $\mathbb
Z\times \mathbb Z$.  Then $(x\cup Fx)^*$ lets us move to the right or add an
element of $F$ translated to the current position and then move right. The term
$y(x^{-1})^*$ allows us to move up one row and then move as far left as
needed.  Now we keep repeating until we are done translating elements of $F$
in positions.  Then we use $\{x^{\pm 1},y^{\pm 1}\}^*$ to return
to the origin.  Notice that when following this procedure, a position
can have at most one element of $F$ translated to it.

For instance, suppose
$m=x^{i_1}y^{j_1}f_{i_1,j_1}+\cdots+x^{i_k}y^{j_k}f_{i_k,j_k}$ is a subset sum
of $F$ where $(i_1,j_1)<(i_2,j_2)<\cdots<(i_k,j_k)$ in right lexicographical
order (i.e., $(a,b)<(c,d)$ if $b<d$, or $b=d$ and $a < c$).  Then we begin
with the product $x^{i_1}y^{j_1}$ from $\{x^{\pm 1},y^{\pm 1}\}^*$ to get to
the starting point of our sum.  Then using $[(x\cup Fx)^*y(x^{-1})^*]^*$ we
build up row by row, always going upward, an element of the form $(m,(a,b))$.
Finally we multiply by  $x^{-a}y^{-b}\in \{x^{\pm 1},y^{\pm 1}\}^*$ to obtain
$(m,0)$.  Conversely, any element of the form $(m,0)$ belonging to $L$ must
have $m$ a subset sum of $F$
since the regular expression $L$ never permits you to translate
by the same element of $\mathbb Z\times \mathbb Z$ twice.
\qed
\end{proof}

Now we can argue as before to obtain undecidability for the two-dimensional lamplighther groups.

\begin{theorem}\label{thm lamplighter}
Rational subset membership is undecidable for a fixed rational
subset of $\mathbb Z/n\mathbb Z\wr (\mathbb Z\times \mathbb Z)$ for any $n\geq 2$.
\end{theorem}

\begin{proof}
Again write $R=\mathbb Z/n\mathbb Z$.  Then $\mathbb Z/n\mathbb Z\wr (\mathbb
Z\times \mathbb Z)=R[\mathbb Z\times \mathbb Z]\rtimes (\mathbb Z\times
\mathbb Z)$.  By Theorem~\ref{subsetsumundec}, there is a free $R[\mathbb
Z\times \mathbb Z]$-module $M$ of rank $m$ with an undecidable subset sum
problem for a fixed finite subset $F$.  Since $R[\mathbb Z\times \mathbb Z]$
is a free $R[m\mathbb Z\times \mathbb Z]$-module of rank $m$, we can embed
$M\rtimes (\mathbb Z\times \mathbb Z)$ in
$\mathbb Z/n\mathbb Z\wr (\mathbb Z\times \mathbb Z)$.
The result now follows from Proposition~\ref{prop undecrat}.
\qed
\end{proof}

As a corollary, it follows that $G\wr (\mathbb Z\times \mathbb Z)$ has an
undecidable rational subset membership problem for any non-trivial group $G$.

\begin{corollary}
 Let $G$ be a non-trivial group.  Then $G\wr (\mathbb Z\times \mathbb Z)$
has an undecidable rational subset membership problem for a fixed rational subset.
\end{corollary}

\begin{proof}
Either $G\wr (\mathbb Z\times \mathbb Z)$ contains a copy of
$\mathbb Z\wr (\mathbb Z\times \mathbb Z)$ or of $\mathbb Z/n\mathbb Z\wr
(\mathbb Z\times \mathbb Z)$.
\qed
\end{proof}

The argument of Theorem~\ref{thm lamplighter} can be adapted to show that membership is undecidable for a fixed rational subset of the free group of rank 2 in the variety $\mathfrak{A}(n) \cdot \mathfrak{A}$, where
$\mathfrak{A}(n)$ (resp. $\mathfrak{A}$) is the variety of abelian groups
of exponent $n$ (resp. of all abelian groups).  The adaptations are entirely analogous to those used in going from submonoid membership for $\mathbb Z\wr (\mathbb Z\times \mathbb Z)$ to submonoid membership for the free metabelian group of rank $2$.

\bibliographystyle{abbrv}
\bibliography{bib}

\end{document}